\documentclass[11pt,twoside]{amsart}
 
\usepackage{anysize} \marginsize{1.3in}{1.3in}{1in}{1in}
\usepackage{amsmath, amsthm, amssymb, hyperref, color}
\usepackage{tikz}
\usepackage{tikz-cd}
\usepackage{stmaryrd}
\usepackage{graphicx}
\usepackage{caption}
\usepackage{mathtools}
\usepackage{enumerate}
\usepackage{listings}
\usepackage[all]{xypic}
\usepackage{verbatim}
\usepackage{mathrsfs}
\usepackage[linesnumbered,lined,commentsnumbered]{algorithm2e}
\usepackage{caption}

\usepackage{makecell}
\usepackage{array}

\newtheorem{manualtheoreminner}{Theorem}
\newenvironment{manualtheorem}[1]{%
  \renewcommand\themanualtheoreminner{#1}%
  \manualtheoreminner
}{\endmanualtheoreminner}

\theoremstyle{theorem}
\newtheorem{theorem}{Theorem}
\newtheorem{proposition}[theorem]{Proposition}
\newtheorem{lemma}[theorem]{Lemma}
\newtheorem{corollary}[theorem]{Corollary}
\theoremstyle{definition}
\newtheorem{definition}[theorem]{Definition}

\newtheorem{remark}[theorem]{Remark}

\numberwithin{theorem}{section}

\newcommand{\CC}{\mathbb{C}}
\newcommand{\ZZ}{\mathbb{Z}}

\DeclareMathOperator{\lm}{lm}
\DeclareMathOperator{\lc}{lc}
\DeclareMathOperator{\lexp}{lexp}
\DeclareMathOperator{\plm}{plm}
\DeclareMathOperator{\plc}{plc}
\DeclareMathOperator{\lcm}{lcm}
\DeclareMathOperator{\Ann}{Ann}
\DeclareMathOperator{\gr}{gr}

\DeclareMathOperator*{\supp}{supp}

\DeclareMathOperator{\Ext}{Ext}
\DeclareMathOperator{\Chr}{Ch^{rel}}

\let\OldMarginpar\marginpar
\renewcommand{\marginpar}[1]{\OldMarginpar{\footnotesize#1}}

\tikzset{
  symbol/.style={
    draw=none,
    every to/.append style={
      edge node={node [sloped, allow upside down, auto=false]{$#1$}}}
  }
}

\begin{document}
\begin{abstract}
In this article we investigate the fibers of relative $D$-modules. In general we prove that there exists an open, Zariski dense subset of the vanishing set of the annihilator over which the fibers of a cyclic relative $D$-module are non-zero.

Next we restrict our attention to relatively holonomic $D$-modules. For this class we prove that the fiber over every point in the vanishing set of the annihilator is non-zero. 

As a consequence we obtain new proofs of a conjecture of Budur which was recently proven by Budur, van der Veer, Wu and Zhou, as well as a new proof of a theorem of Maisonobe. Moreover, we also obtain a diagonal specialization result for Bernstein-Sato ideals.
\end{abstract}
\title{On the support of relative $D$-modules}

\author{Robin van der Veer}
\address{Robin van der Veer\\KU Leuven\\Department of Mathematics\\Celestijnenlaan 200B bus 2400\\B-3001 Leuven, Belgium}
\email{robin.vanderveer@kuleuven.be}

\makeatletter
\@namedef{subjclassname@2020}{
	\textup{2020} Mathematics Subject Classification}
\makeatother


\maketitle

\tableofcontents

\section{Introduction}
Let $F=(f_1,\dots,f_p)$ be a tuple of polynomials on $X=\CC^n$. The Bernstein-Sato ideal of $F$, denoted $B_F$, is the ideal in $\CC[s_1,\dots,s_p]$ consisting of those $b\in\CC[s_1,\dots,s_p]$ for which there exists a $P\in D_n[s_1,\dots,s_p]$ such that
$$bf_1^{s_1}\dots f_p^{s_p}=Pf_1^{s_1+1}\dots f_p^{s_p+1},$$
where $D_n$ is the Weyl algebra. Associated to $F$ we also have the specialization complex $\psi_F(\CC_X)$ of \cite{sab}. This is a generalisation of the nearby cycles complex to a tuple of functions. It comes equipped with $p$ simultaneous monodromy actions, and we denote by $\mathcal{S}(F)\subset (\CC^*)^p$ the support of this monodromy action. We refer to \cite{B}, \cite{BVWZ} for more details and background on $B_F$ and $\psi_F(\CC_X)$.

The motivation of this article is to give an alternative proof of the following theorem:
\begin{manualtheorem}{A}[Conjectured in \cite{B}, proven in \cite{BVWZ}]\label{conj}
Denote by $\text{Exp}:\CC^p\to (\CC^*)^p$ the coordinate-wise exponential map. Then
$$\text{Exp}(Z(B_F))=\mathcal{S}(F).$$
\end{manualtheorem}
In the above theorem and throughout this article $Z(I)$ denotes the scheme defined by an ideal $I$.
The inclusion 
$$\text{Exp}(Z(B_F))\supset\mathcal{S}(F)$$
was proven already in \cite{B}. To analyse the reverse inclusion the following criterion can be extracted from the proof of \cite[Proposition 1.7]{B}.
\begin{proposition}[\cite{B}]\label{crit}
If $\alpha\in Z(B_F)$ and
\begin{equation}\frac{D_n[s_1,\dots,s_p]f_1^{s_1}\dots f_p^{s_p}}{D_n[s_1,\dots,s_p]f_1^{s_1+1}\dots f_p^{s_p+1}}\otimes_{\CC[s_1,\dots,s_p]}\frac{\CC[s_1,\dots,s_p]}{\mathfrak{m}_\alpha}\not=0\label{goal},\end{equation}
then $Exp(\alpha)\in\mathcal{S}(F)$.
\end{proposition}
Since $\mathcal{S}(F)$ is a closed subset of $(\CC^*)^p$, Theorem \ref{conj} follows if we can prove \eqref{goal} for $\alpha$ in an open Zariski dense subset of $Z(B_F)$. This observation motivates the following theorem:
\begin{manualtheorem}{B}\label{mainthm}
Let $\mathfrak{J}\subset D_n[s_1,\dots,s_p]$ be a left ideal, where $D_n$ is the Weyl algebra. Let $\mathfrak{p}$ be a minimal prime divisor of $\mathfrak{J}\cap \CC[s_1,\dots,s_p]$. Then there exists a polynomial $h\in \CC[s_1,\dots,s_p]\setminus \mathfrak{p}$ such that for all $\alpha\in Z(\mathfrak{p})\setminus Z(h)$ with maximal ideal $\mathfrak{m}_\alpha$,
$$\left(\frac{D_n[s_1,\dots,s_p]}{\mathfrak{J}}\right)\otimes_{\CC[s_1,\dots,s_p]}\left(\frac{\CC[s_1,\dots,s_p]}{\mathfrak{m}_\alpha}\right)\not=0.$$
\end{manualtheorem}
As mentioned above, Theorem \ref{conj} was recently proven in \cite{BVWZ} where it was deduced from the following theorem.
\begin{manualtheorem}{C}[\cite{BVWZ}]\label{thm}
Let $F=(f_1,\dots,f_p)$ be a tuple of polynomials on $\CC^n$. For every codimension $1$ irreducible component $H$ of $Z(B_F)$ there is a Zariski open subset $V\subset H$ such that for every $\alpha\in V$ with maximal ideal $\mathfrak{m}_\alpha$,
$$\frac{D_n[s_1,\dots,s_p]f_1^{s_1}\dots f_p^{s_p}}{D_n[s_1,\dots,s_p]f_1^{s_1+1}\dots f_p^{s_p+1}}\otimes_{\CC[s_1,\dots,s_p]}\frac{\CC[s_1,\dots,s_p]}{\mathfrak{m}_\alpha}\not=0.$$
\end{manualtheorem}
This theorem is a special case of Theorem \ref{mainthm}, since it only deals with codimension $1$ components. 
As explained above, it follows from Proposition \ref{crit} and this theorem that for every codimension $1$ component $H$ of $Z(B_F)$, $Exp(H)\subset \mathcal{S}(F)$. This restriction of only working with codimension $1$ components meant that the following result from \cite{Mai16} had to be used to prove Theorem \ref{conj} in \cite{BVWZ}:
\begin{manualtheorem}{D}[\cite{Mai16}]\label{translate}
Every irreducible component of $Z(B_F)$ can be translated along an integer vector into a codimension $1$ component of $Z(B_F)$.
\end{manualtheorem}
Theorem \ref{translate} is a corollary of Theorem \ref{conj}, since it is known from \cite[Theorem 1.3]{budur2017} that $\mathcal{S}(F)$ is a union of codimension one torsion translated subtori of $(\CC^*)^p$. Since our proof of Theorem \ref{conj} does not rely on Theorem \ref{translate}, our methods provide a new proof of Theorem \ref{translate}.
Let us also emphasise that, unlike in the methods of \cite{BVWZ}, in Theorem \ref{mainthm} no assumption on relative holonomicity of $D_n[s_1,\dots,s_p]/\mathfrak{J}$ is made.

The proof strategy for Theorem \ref{mainthm} is to specialize a suitable Gr\"obner basis for $\mathfrak{J}+\mathfrak{p}$ to a Gr\"obner basis for $\left(D_n[s_1,\dots,s_p]/\mathfrak{J}\right)\otimes_{\CC[s_1,\dots,s_p]}\left(\CC[s_1,\dots,s_p]/\mathfrak{m}_\alpha\right)$. Using a result on specializing Gr\"obner bases from \cite{LEYKIN} and \cite{OAKU} we can then conclude Theorem \ref{mainthm}. This method of proof works in every ring where Gr\"obner basis methods are available, and in particular does not use any $D_n$-module specific constructions. 

In the fourth section we do bring $D_n$-module theoretical constructions into focus, and restrict our attention to relatively holonomic $D_n[s_1,\dots,s_p]$-modules (see Definition \ref{relhol}). We then obtain the following improvement of Theorem \ref{mainthm}
\begin{manualtheorem}{E}\label{relholmain}
Let $M$ be a relatively holonomic $D_n[s_1,\dots,s_p]$-module. Then for all $\alpha\in Z(\Ann_{\CC[s_1,\dots,s_p]}(M))$ with maximal ideal $\mathfrak{m}_\alpha$,
$$M\otimes_{\CC[s_1,\dots,s_p]} \frac{\CC[s_1,\dots,s_p]}{\mathfrak{m}_\alpha}\not=0.$$
\end{manualtheorem}
In the final section we apply Theorem \ref{relholmain} to investigate the diagonal specialization of Bernstein-Sato ideals. By this we mean the following: denote by $\Delta=\{s_1=\dots=s_p\}$ the diagonal in $\CC^p$, and $f=\prod_{i=1}^pf_i$. It is then obvious that
$$Z(B_F)\cap \Delta \supset Z(b_f),$$
where $b_f$ is the usual Bernstein-Sato polynomial of $f$, and the affine space on which $b_f$ lives is identified with $\Delta$ in the obvious way. The reverse inclusion is a delicate matter, and it is expected that it fails in general, although it proved difficult to exhibit this failure with a concrete example. We are able to prove the following partial positive answer to the reverse inclusion:
\begin{manualtheorem}{F}\label{spec}
Let $F=(f_1,\dots,f_p)$ be a tuple of polynomials on $\CC^n$ and denote $f=\prod_{i=1}^pf_i$. Denote by $\nabla:\CC^n\to \CC^n$ the map $\nabla(\alpha_1,\dots,\alpha_p)=(\alpha_1+1,\dots,\alpha_p+1)$ and $\Delta=\{s_1=\dots=s_p\}\subset \CC^n$. Then
$$(Z(B_F)\cap \Delta) \setminus \bigcup_{i>0}\nabla^i(Z(B_F))\subset Z(b_f),$$
where the right hand side is identified with a subset of $\Delta$.
\end{manualtheorem}
This specialization result makes no assumption on the tuple $F$. For related conditional results on diagonal specialization we refer to \cite{2020arXiv200807447B} and \cite{2020arXiv200513502W}.

\subsubsection*{Acknowledgement}
We would like to thank Guillem Blanco, Nero Budur, Feng Hao, Alexander Van Werde and Lei Wu for the helpful comments and suggestions. 

The author is supported by a PhD Fellowship of the Research Foundation - Flanders.

\section{A lemma in commutative algebra, and a non-commutative corollary}
We denote $A=\CC[s_1,\dots,s_p]$. An ideal $\mathfrak{q}\subset A$ is called \textit{primary} if for all $x,y\in A$ with $xy\in \mathfrak{q}$, either $x\in \mathfrak{q}$ or $y\in \sqrt{\mathfrak{q}}$. If $\mathfrak{q}$ is primary, then $\mathfrak{p}=\sqrt{\mathfrak{q}}$ is a prime ideal. When we say that $\mathfrak{q}$ is $\mathfrak{p}$-primary we mean that $\mathfrak{q}$ is primary and $\sqrt{q}=\mathfrak{p}$. 

\begin{lemma}\label{lemma:sec3}
Let $\mathfrak{q}\subset A$ be a $\mathfrak{p}$-primary ideal. If $\mathfrak{q}\not=\mathfrak{p}$, then there exists an $f\in \mathfrak{p}\setminus \mathfrak{q}$ such that $f\mathfrak{p}\subset \mathfrak{q}$.
\end{lemma}
\begin{proof}
Let $N\subset A/\mathfrak{q}$ be the nilradical, which is non-zero since $\mathfrak{q}\not=\mathfrak{p}$. Let $g_1,\dots,g_m\in N$ be a set of generators. Let $f_0=1\in A/\mathfrak{q}$. Using induction we define for $i=1,\dots,m$:
$$f_{i+1}=f_ig_{i+1}^{k_{i+1}-1},$$
where $k_{i+1}\in \ZZ$ is the smallest integer for which
$$f_{i}g_{i+1}^{k_{i+1}}=0.$$
Notice that such $k_{i+1}$ always exists, since each $g_i$ is nilpotent and that $k_{i+1}$ is always at least $1$ since $f_{i}$ is not zero, by induction. By construction we have that for each $i=1,\dots,m$, 
$$g_if_m=0.$$
Let $f$ be a lift of $f_m$ to $A$. Since $f_m\not=0$, $f\not\in \mathfrak{q}$, and since $f_m$ is nilpotent, $f\in\mathfrak{p}$. Let $p\in \mathfrak{p}$. Then in $A/\mathfrak{q}$ we can write $p+\mathfrak{q}=\sum_{i=1}^m a_ig_i$ for some $a_i\in A/\mathfrak{q}$. Then 
$$fp+\mathfrak{q}=\sum_{i=1}^ma_ig_if_m=0,$$
so that $fp\in \mathfrak{q}$.
\end{proof}

Every ideal $\mathfrak{J}\subset A$ has a \textit{primary decomposition}. This means that we can write $\mathfrak{J}=\bigcap_{i=1}^m\mathfrak{q}_i$ such that
\begin{enumerate}
\item every $\mathfrak{q}_i$ is primary, and
\item for all $1\leq j\leq m$, $\bigcap_{i=1}^m\mathfrak{q}_i\subsetneq \bigcap_{\substack{i=1\\i\not=j}}^m\mathfrak{q}_i$, and
\item The prime ideals $\sqrt{\mathfrak{q}_1},\dots,\sqrt{\mathfrak{q}_m}$ are pairwise distinct.
\end{enumerate}
The minimal (under the inclusion order) elements of the set $\{\sqrt{\mathfrak{q}_i}\}$ are uniquely determined by $\mathfrak{J}$, and are called the \textit{minimal prime divisors of }$\mathfrak{J}$.

\begin{theorem}\label{techThm}
Let $\mathfrak{J}\subset A$ be an ideal with primary decomposition $\mathfrak{J}=\bigcap_{i=1}^m \mathfrak{q}_i$, $m>1$. Let $\sqrt{\mathfrak{q}_j}$ be a minimal prime divisor.
Then there exists an $h\in A\setminus \mathfrak{q}_j$ such that $h\sqrt{\mathfrak{q}_j}\subset \mathfrak{J}$.
\end{theorem}
\begin{proof}
We assume without loss of generality that $j=1$. We claim that $\bigcap_{j=2}^m\mathfrak{q}_i\not\subset \sqrt{\mathfrak{q}_1}$. If we would have $\bigcap_{j=2}^m\mathfrak{q}_i\subset \sqrt{\mathfrak{q}_1}$ then there is some $\mathfrak{q}_i$ contained in $\sqrt{\mathfrak{q}_1}$, since the latter is prime. Hence also $\sqrt{\mathfrak{q}_i}\subset \sqrt{\mathfrak{q}_1}$. Since $\sqrt{\mathfrak{q}_1}$ is a minimal prime, this must be an equality. However, by definition of the primary decomposition, $\sqrt{\mathfrak{q}_i}\not=\sqrt{\mathfrak{q}_1}$, and this contradiction proves the claim. Let $g\in \bigcap_{i=2}^m \mathfrak{q}_i\setminus \sqrt{\mathfrak{q}_1}$.

If $\sqrt{\mathfrak{q}_1}=\mathfrak{q}_1$, then $h=g$ satisfies the condition of the theorem. Namely, let $q\in \mathfrak{q}_1$. Then $gq\in \left(\bigcap_{i=2}^m \mathfrak{q}_i\right)\mathfrak{q}_1\subset  \bigcap_{i=1}^m \mathfrak{q}_i=\mathfrak{J}$.

If $\sqrt{\mathfrak{q}_1}\not=\mathfrak{q}_1$, we get from Lemma \ref{lemma:sec3} an $f\in \sqrt{\mathfrak{q}_1}\setminus \mathfrak{q}_1$ such that $f\sqrt{\mathfrak{q}_1}\subset \mathfrak{q}_1$. Set $h=fg$. We claim that this $h$ satisfies the conditions of the theorem. If $h\in \mathfrak{q}_1$, then since $\mathfrak{q}_1$ is primary, $f\in \mathfrak{q}_1$ or $g\in \sqrt{\mathfrak{q}_1}$, neither of which is possible by choice of $f$ and $g$. This means that indeed $h\in A\setminus \mathfrak{q}_1$. Let $p\in \sqrt{\mathfrak{q}_1}$. By choice of $f$, $pf\in \mathfrak{q}_1$, and thus $ph=gpf\in (\bigcap_{i=2}^m\mathfrak{q}_i)\mathfrak{q}_1\subset \bigcap_{i=1}^m\mathfrak{q}_i=\mathfrak{J}$, which concludes the proof.
\end{proof}

\begin{corollary}\label{sec3:cor}
Let $M$ be an $A$-module, and let $\mathfrak{p}$ be a minimal prime divisor of the ideal $\Ann_A(M)$. Then $\Ann_A(M\otimes_A(A/\mathfrak{p}))=\mathfrak{p}$.
\end{corollary}
\begin{proof}
Let $\Ann_A(M)=\bigcap_{i=1}^m\mathfrak{q}_i$ be a primary decomposition of $\Ann_A(M)$ with $\sqrt{\mathfrak{q}_1}=\mathfrak{p}$. From Theorem \ref{techThm} we get an $h\in A\setminus \mathfrak{q}_1$ such that $h\mathfrak{p}\subset \Ann_A(M)$.

Let $g\in \Ann_{A}(M\otimes_A(A/\mathfrak{p}))$. This means that for every $m\in M$, $gm\in \mathfrak{p}M$. In other words, for every $m\in M$, there exist $n_1,\dots,n_k\in M$ and $p_1,\dots,p_k\in \mathfrak{p}$ such that $gm=\sum_{i=1}^kp_in_i$. We multiply this equation by $h$ on both sides to find $ghm=\sum_{i=1}^khp_in_i$. By construction, $hp_i\in \Ann_A(M)$, so that $ghm=0$ for all $m$, and thus $gh\in \Ann_A(M)$. In particular, $gh\in \mathfrak{q}_1$, so that either $h\in \mathfrak{q}_1$ or $g\in \sqrt{\mathfrak{q}_1}=\mathfrak{p}$. Since $h\not\in \mathfrak{q}_1$ by construction, we conclude that $g\in \mathfrak{p}$, which shows that  $\Ann_{A}(M\otimes_A(A/\mathfrak{p}))\subset \mathfrak{p}$. The other inclusion is obvious, and this concludes the proof.
\end{proof}

\begin{lemma}\label{surprise}
Let $\mathfrak{J}\subset D_n[s_1,\dots,s_p]$ be a left ideal. Let $\mathfrak{p}$ be a minimal prime divisor of $\mathfrak{J}\cap A$. Then $(\mathfrak{J}+R\mathfrak{p})\cap A=\mathfrak{p}$.
\end{lemma}
\begin{proof}
We denote $R=D_n[s_1,\dots,s_p]$. We regard $R/\mathfrak{J}$ as an $A$-module. As such we claim that $\Ann_{A}(R/\mathfrak{J})=\mathfrak{J}\cap A$. To see this, let $f\in \Ann_{A}(R/\mathfrak{J})$, so that $f\cdot (1+\mathfrak{J})=f+\mathfrak{J}$ is zero in $R/\mathfrak{J}$, which means that $f\in \mathfrak{J}\cap A$. For the other inclusion let $f\in \mathfrak{J}\cap A$. Then 
$$f\cdot(P+\mathfrak{J})=fP+\mathfrak{J}=Pf+\mathfrak{J},$$
where in the second equality we use that $A$ is contained in the center of $R$. Since $f\in \mathfrak{J}$, which is a left ideal, we conclude that $Pf\in\mathfrak{J}$, so that $f\cdot(P+\mathfrak{J})=0$ and hence indeed $f\in \Ann_A(R/\mathfrak{J})$.

We apply Corollary \ref{sec3:cor} to find that $\Ann_A((R/\mathfrak{J})\otimes_A(A/\mathfrak{p}))=\mathfrak{p}$. The lemma  follows when we prove that $\Ann_A((R/\mathfrak{J})\otimes_A(A/\mathfrak{p}))=(\mathfrak{J}+R\mathfrak{p})\cap A$, which in turn follows, as above, when we prove that we have an isomorphism
$$(R/\mathfrak{J})\otimes_A(A/\mathfrak{p})\cong R/(\mathfrak{J}+R\mathfrak{p}).$$
There is a well-defined map from the left hand side to the right hand side given by
$$(P+\mathfrak{J})\otimes(f+\mathfrak{p})\mapsto Pf+(\mathfrak{J}+R\mathfrak{p}).$$ 
Note that this is well-defined because $A$ is contained in the center of $R$. In the other direction we have the map
$$P+(\mathfrak{J}+R\mathfrak{p}) \mapsto (P+\mathfrak{J})\otimes(1+\mathfrak{p}).$$
These maps are inverse to each other, which proves the claim.
\end{proof}

\section{Gr\"obner bases for $D_n$-modules}
In this section we recall some facts about Gr\"obner bases for ideals in $D_n$ and $R=D_n[s_1,\dots,s_p]$, and then prove Theorem \ref{mainthm}. We denote as in the previous section $A=\CC[s_1,\dots,s_p]$. We will always implicitly regard $A$ and $D_n$ as subsets of $R$.

In $D_n$ we consider the set of standard monomials of the form $x^\alpha\partial^\beta$, $\alpha,\beta\in\ZZ_{\geq 0}^n$. On these monomials we consider the lexicographical order with
\begin{equation}\label{DnOrder}\partial_1>\dots>\partial_n>x_n>\dots>x_1.\end{equation}
Any operator in $D_n$ is a finite $\CC$-linear combination of standard monomial in a unique way. Writing $P\in D_n$ as a linear combination $P=\sum_{\alpha,\beta}c_{\alpha,\beta}x^\alpha\partial^\beta$ with non-zero  $c_{\alpha,\beta}$ we denote by $\lm_{D_n}(P)$ the largest monomial $x^\alpha\partial^\beta$ with respect to the lexicographical order. For an ideal $\mathfrak{I}\subset D_n$ we denote $\lm_{D_n}(\mathfrak{I})=\{\lm_{D_n}(P)\mid P\in \mathfrak{I}\}$. A finite generating set  $G\subset \mathfrak{I}$ is called a Gr\"obner basis for $\mathfrak{I}$ if the following is true: for every $m\in \lm_{D_n}(\mathfrak{I})$ there exists a $P\in G$ such that $\sigma(\lm_{D_n}(P))\mid\sigma(m)$ where $\sigma$ denotes the operation of taking the principal symbol. More explicitly this says that when $x^{\alpha}\partial^\beta\in \lm_{D_n}(\mathfrak{I})$ there exists a $P\in G$ with $\lm_{D_n}(P)=x^{\alpha'}\partial^{\beta'}$ and $(\alpha',\beta')$ is entry-wise less than or equal to $(\alpha,\beta)$. To simply the notation we will write this divisibility notion simply as $x^{\alpha'}\partial^{\beta'}\mid x^{\alpha}\partial^{\beta}$, but we emphasise that this does not mean that there exists some $P\in D_n$ for which $Px^{\alpha'}\partial^{\beta'}=x^{\alpha}\partial^{\beta}$. It follows immediately from the definition that $\mathfrak{I}=D_n$ if and only if any Gr\"obner basis for $\mathfrak{I}$ contains a unit, i.e. an element of $\CC$. 

In $R$ we consider the set of standard monomials of the form $x^\alpha\partial^\beta s^\gamma$, $\alpha,\beta\in \ZZ_{\geq 0}^n$, $\gamma\in \ZZ_{\geq 0}^p$. On these monomials we consider the lexicographical order with
\begin{equation}\label{ROrder}\partial_1>\dots>\partial_n>x_n>\dots>x_1>s_p>\dots>s_1.\end{equation}
Any operator in $R$ is a finite $\CC$-linear combination of standard monomial in a unique way. We denote by $\lm_{R}(P)$ the leading monomial of an operator $P\in R$  with respect to the lexicographical order. For an ideal $\mathfrak{Q}\subset R$ we denote $\lm_R(\mathfrak{Q})=\{\lm_{R}(P)\mid P\in \mathfrak{Q}\}$. A finite generating set  $G\subset \mathfrak{Q}$ is called a Gr\"obner basis for $\mathfrak{Q}$ if the following \textit{Gr\"obner property} is true: for every $m\in \lm_{R}(\mathfrak{Q})$ there exists a $P\in G$ such that $\sigma(\lm_{R}(P))\mid\sigma(m)$ where $\sigma$ denotes the operation of taking the principal symbol. More explicitly this says that when $x^{\alpha}\partial^\beta s^\gamma\in \lm_{R}(\mathfrak{Q})$ there exists a $P\in G$ with $\lm_{R}(P)=x^{\alpha'}\partial^{\beta'}s^{\gamma'}$ and $(\alpha',\beta',\gamma')$ is entry-wise less than or equal to $(\alpha,\beta,\gamma)$. Again we denote this divisibility relation simply by $x^{\alpha'}\partial^{\beta'}s^{\gamma'}\mid x^{\alpha}\partial^\beta s^\gamma$.
Regarding the $s_1,\dots,s_p$ as parameters we will also need the following. Any $P\in R$ can be written uniquely as
\begin{equation}\label{P}P=\sum_{\alpha,\beta}h_{\alpha,\beta}(s_1,\dots,s_p)x^{\alpha}\partial^{\beta},\end{equation}
with $h_{\alpha,\beta}\not=0$.
 We denote the \textit{parametric leading monomial} of $P$ by $\plm(P)=x^\alpha\partial^\beta$ and the \textit{parametric leading coefficient} op $P$ by  $\plc(P)=h_{\alpha,\beta},$ where $x^{\alpha}\partial^\beta$ is the largest monomial occurring in \eqref{P} with respect to the monomial order \eqref{DnOrder}.
 
 In the commutative ring $A$ consider the lexicographical monomial order with
\begin{equation}\label{AOrder}s_p>\dots>s_1.\end{equation}
We denote by $\lm_{A}(f)$ the leading monomial of $f\in A$. 

 In all three rings $D_n,A$ and $R$ a Gr\"obner basis for an ideal can be obtained by applying Buchbergers algorithm to an arbitrary generating set. 
 We note the following relations between the leading monomials. For $f\in A$: $\lm_R(f)=\lm_A(f)$. For $P\in D_n$, $\lm_{D_n}(P)=\lm_R(P)$. For $P\in R$, $\lm_A(\plc(P))\plm(P)=\lm_R(P)$. Because of these equalities we will from now on suppress the notation of the ring and simply write $\lm$ for leading monomials.
 
 \begin{lemma}\label{intersectionGB}
Let $\mathfrak{Q}\subset R$ be an ideal and let $G$ be a Gr\"obner basis for $\mathfrak{Q}$ with respect to the order \eqref{ROrder}. Then $G\cap A$ is a Gr\"obner basis for $\mathfrak{Q}\cap A$ with respect to the order \eqref{AOrder}.
 \end{lemma}
 \begin{proof}
 Let $f\in \mathfrak{Q}\cap A$. There exists a $P\in G$ such that $\lm(P)\mid \lm(f)$. This means that $\lm(P)\in A$, and thus by definition of the order \eqref{ROrder}, $P\in A$. We conclude that for every $f\in \mathfrak{Q}\cap A$ there exists a $g\in G\cap A$ such that $\lm(g)\mid \lm(f)$. It only remains to show that $G\cap A$ generates $\mathfrak{Q}\cap A$. This follows from the preceding statement, since we can reduce $f$ to zero modulo $G\cap A$ by iteratively canceling leading monomials.
 \end{proof}
 \begin{lemma}\label{reducedGB}
 Let $\mathfrak{Q}\subset R$ be an ideal. Then there exists a Gr\"obner basis $G$ for $\mathfrak{Q}$ with respect to the order \eqref{ROrder} such that for all $P\in G\setminus A$, $\plc(P)\not\in \mathfrak{Q}\cap A$.
 \end{lemma}
 \begin{proof}
Let $P\in G\setminus A$ and write
$$P=\plc(P)\plm(P)+Q.$$
 By Lemma \ref{intersectionGB} $G\cap A=\{f_1,\dots,f_m\}$ is a Gr\"obner basis for $\mathfrak{Q}\cap A$. Using the division algorithm \cite[Proposition 1.6.1]{IVA} we find $q_1,\dots,q_m,r$ such that
 $$\plc(P)=\sum_{i=1}^mq_if_i+r,$$
 where $\lm(r)$ is not divisible by any $\lm(f_i)$. This means that
 $$P=\left(\sum_{i=1}^mq_if_i+r\right)\plm(P)+Q\equiv r\plm(P)+Q=:P',$$
 where $\equiv$ denotes equivalence modulo $\mathfrak{Q}$.
 If $P'=0$ we define $G'=G\setminus \{P\}$ and if $P'\not=0$ then we define $G'=G\setminus\{P\}\cup \{P'\}$.
 
 We claim that $G'$ is still a Gr\"obner basis for $\mathfrak{Q}$. This is clear if $P=P'$ so assume that $P\not=P'$. We first remark that $G'$ clearly still generates $\mathfrak{Q}$, and hence we need only verify the Gr\"obner property. If $\lm(P')=\lm(P)$ then the Gr\"obner property is clearly still satisfied, so assume that $\lm(P')\not=\lm(P)$. This assumption means that $\lm(\plc(P))$ was divisible by $\lm(f_i)$ for some $f_i$. To see this, note that if this were not the case then by the definition of the division algorithm, $\lm(r)=\lm(\plc(P))$, and hence $\lm(P)=\lm(P')$, contradicting our assumption. This means that $\lm(P)=\lm(\plc(P))\plm(P)$ is divisible by some $\lm(f_i)$. Hence any element $Q\in \mathfrak{Q}$ for which $\lm(P)\mid \lm(Q)$ also satisfies $\lm(f_i)\mid \lm(Q)$. It follows that $G'$ is still a Gr\"obner basis, since $P$ was redundant for the Gr\"obner basis property.
  
  We keep repeating this procedure until no $P\in G\setminus A$ is changed by applying this reduction. This is clearly a finite process. Once we are done we thus find that for all $P\in G\setminus A$, $\lm(\plc(P))$ is not divisible by any $\lm(f_i)$. Since the $\{f_1,\dots,f_m\}$ form a Gr\"obner basis for $\mathfrak{Q}\cap A$, this means that for all $P\in G\setminus A$, $\plc(P)\not\in \mathfrak{Q}\cap A$.
 \end{proof}

For any $\alpha\in \CC^p$ there is a specialization map $q_\alpha:R\to D_n$ defined by
 $$q_\alpha(P(s_1,\dots,s_p))=P(\alpha_1,\dots,\alpha_p).$$
Notice that for any ideal $\mathfrak{Q}\subset R$ we have
$$(R/\mathfrak{Q})\otimes_A (A/\mathfrak{m}_\alpha) \cong D_n/q_\alpha(\mathfrak{Q}),$$
where $\mathfrak{m}_\alpha$ denotes the maximal ideal corresponding to $\alpha$.
\begin{theorem}[\cite{LEYKIN}, Lemma 2.5]\label{mainRef}
Let $\mathfrak{Q}\subset R$ be an ideal, let $\mathfrak{p}$ be an ideal of the ring $A$ contained in $\mathfrak{Q}$, and let $G\subset \mathfrak{Q}$ be a Gr\"obner basis for $\mathfrak{Q}$ with respect to the order \eqref{ROrder}. Let $h=\prod_{P\in G\setminus \mathfrak{p}}\plc(P)$. Let $\alpha\in Z(\mathfrak{p})\setminus Z(h)$. Then $q_{\alpha}(G\setminus \mathfrak{p})$ is a Gr\"obner basis for $q_{\alpha}(\mathfrak{Q})$ with respect to the order \eqref{DnOrder}.
\end{theorem}
\begin{proof}
Apply \cite{LEYKIN} Lemma 2.5 without any $y$-parameters.
\end{proof}
We remark that a statement similar to Theorem \ref{mainRef} also appeared in \cite{OAKU}, and that the comprehensive Gr\"obner bases from \cite{paramGB} can also be used to give a simple proof of this result.

\begin{proof}[Proof of Theorem \ref{mainthm}]
Let $G$ be a Gr\"obner basis for $\mathfrak{Q}=\mathfrak{J}+R\mathfrak{p}$ with respect to the order \eqref{ROrder}. By Lemma \ref{reducedGB} we can assume that for all $P\in G\setminus A$, 
$$\plc(P)\not\in \mathfrak{Q}\cap A=\mathfrak{p},$$ 
where this equality follows from Corollary \ref{surprise}. 

Since $\mathfrak{Q}\cap A=\mathfrak{p}$, clearly we have $G\cap A\subset \mathfrak{p}$, and hence $G\cap \mathfrak{p}=G\cap A$, and hence also 
$$G\setminus A=G\setminus \mathfrak{p}.$$

Putting the preceding two statements together we conclude that for all $P\in G\setminus\mathfrak{p}$, $\plc(P)\not\in \mathfrak{p}$. Since $\mathfrak{p}$ is prime, this means that 
$$h=\prod_{P\in G\setminus \mathfrak{p}}\plc(P)\not\in \mathfrak{p}$$ 

 Since $G\cap\mathfrak{p}=G\cap A$, we have that $(G\setminus \mathfrak{p})\cap A=\emptyset$, and hence for every $P\in G\setminus \mathfrak{p}$, $\plm(P)\not\in \CC$. By choice of $h$, for all $\alpha \in Z(\mathfrak{p})\setminus Z(h)$ and $P\in G\setminus \mathfrak{p}$, $\plm(P)=\lm(q_\alpha(P))$, so that $q_\alpha(G\setminus\mathfrak{p})$ does not contain any units. Since this set is a Gr\"obner basis for $q_\alpha(\mathfrak{Q})$ by Theorem \ref{mainRef} we conclude that $D_n/q_\alpha(\mathfrak{Q})\not=0$. 

Now notice that
$$\frac{D_n}{q_{\alpha}(\mathfrak{Q})}\cong \frac{R}{\mathfrak{Q}}\otimes_A \frac{A}{\mathfrak{m}_\alpha}\cong \left(\frac{R}{\mathfrak{J}}\otimes_A \frac{A}{\mathfrak{p}}\right)\otimes_A \frac{A}{\mathfrak{m}_\alpha}\cong \frac{R}{\mathfrak{J}}\otimes_A\frac{A}{\mathfrak{m}_\alpha},$$
which proves the claim.
\end{proof}

\begin{proof}[Proof of Theorem \ref{conj}]
We denote 
$$M=\frac{D_n[s_1,\dots,s_p]f_1^{s_1}\dots f_p^{s_p}}{D_n[s_1,\dots,s_p]f_1^{s_1+1}\dots f_p^{s_p+1}}.$$
 Then $B_F=\Ann_{A}(M)$, and $M$ is a cyclic $D_n[s_1,\dots,s_p]$-module. Let $C\subset Z(B_F)$ be an irreducible component. By Theorem \ref{mainthm} there is an $h\in A$ which does not vanish identically on $C$ such that for all $\alpha\in C\setminus Z(h)$, $M\otimes_A (A/ \mathfrak{m}_\alpha)\not=0$. 

By Proposition \ref{crit} we conclude that 
$$\text{Exp}(C\setminus Z(h))\subset \mathcal{S}(F).$$
The map $\text{Exp}:\CC^p\to (\CC^*)^p$ is continuous for the analytic topology. In the analytic topology we have $\overline{C\setminus Z(h)}=C$, and hence
$$\text{Exp}(C)\subset \overline{\text{Exp}(C\setminus Z(h))}\subset\overline{\mathcal{S}(F)}=\mathcal{S}(F),$$
which is what we wanted to prove.
\end{proof}

\section{Relatively holonomic modules}
We keep the notation of the previous section and denote $A=\CC[s_1,\dots,s_p], R=D_n[s_1,\dots,s_p]$.
In this section we investigate non-vanishing of the specialization under the assumption that the $R$-module $M$ is relatively holonomic. To explain what this means we recall the following definitions. On $R$ we have the \textit{relative filtration}
$$F_pR := F_pD_n\otimes_{\CC}\CC[s_1,\dots,s_p],$$
where $F_pD_n$ is the usual order filtration. The associated graded ring can be identified with the ring of regular functions on $(T^*\CC^n)\times \CC^p$. Any coherent $R$-module can be equipped with a good filtration $F_\bullet M$ and this yields a $\gr(R)$-module $\gr(M)$. 
\begin{definition}\label{relhol}
An $R$-module $M$ is called relatively holonomic if 
$$\Chr(M):=\supp_{(T^*\CC^n)\times \CC^p}(\gr(M))=\bigcup_{i\in I}\Lambda_i\times S_i,$$
where the $\Lambda_i$ are Lagrangian subvarieties of $T^*\CC^n$ and the $S_i$ are subvarieties of $\CC^p$.
\end{definition}
We will denote by $\pi_2:(T^*\CC^n)\times\CC^p\to \CC^p$ the second projection.
It is a standard fact that $\Chr(M)$ does not depend on the chosen good filtration, and that it equals $Z(\sqrt{\Ann_{\gr(R)}(\gr(M))})_{red}$. Moreover, the irreducible components of $\Chr(M)$ come equipped with well-defined multiplicities.
\begin{proposition}[\cite{BVWZ}, Lemma 3.4.1]\label{proj}
If $M$ is relatively holonomic, then $\pi_2(\Chr(M))=Z(\Ann_A(M))$.
\end{proposition}
To analyse the specialization of relatively holonomic $R$-modules we will need the notion of purity and grade.
\begin{definition}
Let $M$ be an $R$-module. The \textit{grade} $j(M)$ of $M$ is defined to be the least integer for which $\Ext^i_R(M,R)\not=0$. $M$ is called \textit{pure} if for every non-zero submodule $N\subset M$, $j(N)=j(M)$.
\end{definition}
\begin{proposition}[\cite{bjork}, A:IV.2.3]\label{min}
For a short exact sequence $0\to M_1\to M_2\to M_3\to 0$ we have
$$j(M_2)=\min\{j(M_1),j(M_3)\}.$$
\end{proposition}

\begin{proposition}\label{pure}
If $M$ is pure of grade $k$ then there exists a good filtration $F_\bullet$ on $M$ such that $\gr(M)$ has no embedded associated primes, and $\Chr(M)$ is purely $2n+p-k$ dimensional.
\end{proposition}
\begin{proof}
By \cite[A:IV:4.11]{bjork} there exists a filtration $F_\bullet$ on $M$ such that $\gr(M)$ is pure and $j(\gr(M))=j(M)$. By \cite[A:IV:3.7]{bjork}, $\gr(M)$ does not have embedded primes and $\dim(A_{\mathfrak{p}})=j(M)$ for every minimal prime of $\gr(M)$. This means that $\dim(Z(\mathfrak{p}))=\dim(A)-j(M)=2n+p-j(M)$.
\end{proof}

\begin{corollary}
Let $M$ be a pure $R$-module. Let $f\in A$ be such that $Z(f)$ does not contain a component of $Z(\Ann_A(M))$. Then $f\in A$ is not a zero divisor on $M$.
\end{corollary}
\begin{proof}
Chose a filtration on $M$ such that $\gr(M)$ has no embedded primes, using Proposition \ref{pure}. This means that $f\in A$ is a zero divisor on $gr(M)$ if and only if $Z(f)$ contains a component of $\Chr(M)$. But by Proposition \ref{proj} this is equivalent to $Z(f)$ containing a component of $Z(\Ann_A(M))$. Clearly $f$ not being a zero divisor on $\gr(M)$ implies that it is not a zero divisor on $M$.
\end{proof}

\begin{theorem}[\cite{BVWZ}]\label{non-vanishing}
Suppose that $M$ is pure, and let $f\in A$ be such that $f$ is not contained in any minimal prime ideal containing $Ann_A(M)$. Then 
$$\Chr\left(M\otimes_A \frac{A}{(f)}\right)=\Chr(M)\cap Z(f).$$
In particular, if $Z(f)\cap Z(\Ann_A(M))\not=\emptyset$, then $M\otimes_A \frac{A}{(f)}\not=0$.
\end{theorem}
\begin{remark}
This theorem implies that $M\otimes_A(A/(f))$ will be a relatively holonomic $R$-module, and the characteristic variety will be purely $2n+p-k-1$ dimensional. However, in general it need not be pure, since the associated graded can have embedded primes.
\end{remark}

We define
$$G=\bigcup_{\substack{N\subset M\\ j(N)>j(M)}}N.$$
We refer also to \cite[A:IV:2]{bjork} for a more general and detailed analysis of this construction and the following lemma.
\begin{lemma}\label{gradeG}
$G$ is a submodule of $R$ with $j(G)>j(M)$.
\end{lemma}
\begin{proof}
For the first claim we need only prove that if $g_1,g_2\in G$, then also $g_1+g_2\in G$. It then suffices to show that if $N_1,N_2\subset M$ are such that $j(N_1),j(N_2)>j(M)$, then also $j(N_1+N_2)>j(M)$. 
There is an exact sequence
$$0\to N_1\to N_1+N_2\to \frac{N_2}{N_1\cap N_2}\to 0,$$
so that by Proposition \ref{min}
$$j(N_1+N_2)=\min\{j(N_1),j(N_2/(N_1\cap N_2))\}.$$
But $j(N_2)=\min\{ j(N_1\cap N_2), j(N_2/(N_1\cap N_2))\}$, which shows that $j(N_2/(N_1\cap N_2))\geq j(N_2)$. This means that $j(N_1+N_2)\geq \min\{j(N_1),j(N_2)\}>j(M)$ which is what we wanted to prove.

Since $M$ is finitely generated, there exist finitely many $N_1,\dots,N_q\subset M$ with $j(N_i)>j(M)$ such that $G=\sum_{i=1}^qN_i$. Using induction and the previous paragraph it then follows that $j(G)>j(M)$.
\end{proof}
It follows that $G$ can also be characterised as the largest submodule of $M$ with grade larger than $j(M)$.
\begin{definition}
The \textit{purification} of $M$ is defined to be $M^p:=M/G$.
\end{definition}
\begin{proposition}\label{relcharpure}
$M^p$ is a pure $R$-module of grade $j(M)$. If $\Chr(M)=\bigcup_{i\in I}\Lambda_i\times S_i$, then
$$\Chr(M^p)=\bigcup_{\substack{i\in I\\ \dim(S_i)=2n+p-j(M)}}\Lambda_i\times S_i.$$
In particular, $Z(\Ann_A(M^p))$ is equal to the top-dimensional part of $Z(\Ann_A(M))$.
\end{proposition}
\begin{proof}
$j(M)=j(M^p)$ follows from Lemma \ref{gradeG} and the equality
$$j(M)=\min\{j(G),j(M^p)\}.$$ 
The second claim now follows when we prove purity of $M^p$ by comparing multiplicities of relative characteristic cycles and the last claim then follows from Proposition \ref{proj}. Hence we need to prove that for every non-zero $H\subset M^p$, $j(H)=j(M^p)$. Denote by $\tilde{H}$ the inverse image of $H$ under the natural map $M\to M^p$. Then we have an exact sequence
$$0\to G\to \tilde{H}\to H\to 0.$$
Since $H\not=0$, $G\not=\tilde{H}$, and hence by definition of $G$, $j(\tilde{H})\leq j(M)<j(G)$. Since $j(\tilde{H})=\min\{j(G),j(H)\}$, we conclude that $j(\tilde{H})=j(H)$. It follows that $j(H)\leq j(M)=j(M^p)$. Since $j(M^p)=\min\{j(H),j(M^p/H)\}$, we conclude that $j(H)=j(M^p)$.
\end{proof}

\begin{proof}[proof of theorem \ref{relholmain}]
Let $C\subset Z(\Ann_A(M))$ be an irreducible component containing $\alpha$. Denote by $I_0$ the ideal of functions vanishing on $C$. We define $M_0:=M\otimes_A \frac{A}{I_0}$. 
By Lemma \ref{sec3:cor} we know that $\Ann_A(M_0)=I_0$, so that in particular $M_0\not=0$. $M_0$ is still relatively holonomic. If $I_0=\mathfrak{m}_\alpha$ we are done. 

If not, let $f\in \mathfrak{m}_\alpha\setminus I_0$. We know from Proposition \ref{relcharpure} and Proposition \ref{proj} that $Z(\Ann_A((M_0)^p))=Z(I_0)$. By choice of $f$ and the purity of $(M_0)^p$, Theorem \ref{non-vanishing} applies, so that
 $$(M_0)^p\otimes_A\frac{A}{(f)}\not=0,$$
 and $\Chr((M_0)^p\otimes_AA/(f))=\Chr((M_0)^p)\cap Z(f)$.
Right exactness of the tensor product applied to the surjection $M_0\to (M_0)^p$ means that also
$$M_1:=M_0\otimes_A \frac{A}{(f)}=M\otimes_A\frac{A}{I_0+(f)}\not=0.$$ 
If $I_1:=I_0+(f)=\mathfrak{m}_\alpha$ we are done. 

If not, the surjection $M_1\to (M_0)^p\otimes_A A/(f)$ shows that 
\begin{align*}
Z(Ann_A(M_1))&\supset Z(\Ann_A((M_0)^p\otimes_A A/(f))\\
&=\pi_2(\Chr( (M_0)^p\otimes_A A/(f)))\\
&= \pi_2(\Chr((M_0)^p)\cap Z(f))\ni\alpha.
\end{align*}
We can now iterate the whole previous procedure by applying it to $M_1$ rather than $M$. In this way we keep modding out $M$ by larger and larger ideals, all included in $\mathfrak{m}_\alpha$, and hence this process is finite and eventually proves that $M\otimes_A (A/\mathfrak{m}_\alpha)\not=0$.
\end{proof}

\section{Diagonal specialization of Bernstein-Sato ideals}
We continue to denote $A=\CC[s_1,\dots,s_p]$, $R=D_n[s_1,\dots,s_p]$. Let $f_1,\dots,f_p\in \CC[x_1,\dots,x_n]$, and denote $f=\prod_{i=1}^p f_i$. We denote $\mathcal{O}_X=\CC[x_1,\dots,x_n]$. Then we consider the $R$-module $N_F=\mathcal{O}_X[s_1,\dots,s_p,f^{-1}]f_1^{s_1}\dots f_p^{s_p}$ and the $D_n[s]$-module $N_f=\mathcal{O}_X[s,f^{-1}]f^s$. We denote $F^s=f_1^{s_1}\dots f_p^{s_p}\in N_F$. On $N_F$ there is a $D_n$-linear map $\nabla_F:N_F\to N_F$ which maps
$$g(x_1,\dots,x_n,s_1,\dots,s_p)F^s\mapsto g(x_1,\dots,x_n,s_1+1,\dots,s_p+1)fF^s.$$
Similarly, there is the $D_n$-linear map $\nabla_f:N_f\to N_f$ mapping
$$g(x_1,\dots,x_n,s)f^{s}\mapsto g(x_1,\dots,x_n,s+1)ff^{s}.$$
We denote $M_F = RF^s\subset N_F$, $M_f=D_n[s]f^s\subset N_f$. Clearly $\nabla_F(M_F)\subset M_F$ and $\nabla(M_f)\subset M_f$. 

We denote by $\eta:R\to D_n[s]$ the map
$$P(s_1,\dots,s_p)\mapsto P(s),$$
and by $\eta_F:N_F\to N_f$ the map
$$g(x_1,\dots,x_n,s_1,\dots,s_p)F^s\mapsto g(x_1,\dots,x_n,s,\dots,s)f^s.$$
\begin{proposition}\label{prop5}
\begin{enumerate}
\item \label{prop:1}$\eta$ is a surjective map of rings with kernel equal to $I_\Delta R$, where $I_\Delta$ is the ideal defining the diagonal in $\CC^p$. 
\item \label{prop:2} $\eta_F$ is $R$-linear when $N_f$ is given the structure of an $R$-module via $\eta$.
\item\label{prop:3} $\eta_F$ is surjective, and the kernel is equal to $I_\Delta\mathcal{O}_X[s_1,\dots,s_p,f^{-1}]F^S$, where $I_\Delta$ is the ideal defining the diagonal in $\CC^p$. 
\item\label{prop:4} For every $k>0$, $\eta_F|_{\nabla^k(M_F)}:\nabla^k(M_F)\to N_F$ has image equal to $\nabla_f^k(M_f)$, and kernel equal to $\ker\eta_F\cap \nabla^k(M_F)$.
\item \label{prop:5}$\eta_F$ induces a surjective $R$-linear map $\widetilde{\eta_F}:M_F/\nabla_F(M_F)\to M_f/\nabla_f(M_f)$. The class of $P(s_1,\dots, s_p)F^s$ in $M_F/\nabla_F(M_F)$ maps to zero under  $\widetilde{\eta_F}$ if and only if 
$$P(s_1,\dots,s_p)F^s\in \nabla_F(M_F)+\ker\eta_F\cap M_F$$
\end{enumerate}
\end{proposition}
\begin{proof}
\eqref{prop:1} is clear. For \eqref{prop:2} it is enough to prove that $\eta_F(\partial_i F^S)=\partial_i\eta_F(F^S)$. This is an easy computation. The surjectivity in \eqref{prop:3} is clear. Suppose that 
$$g(x_1,\dots,x_n,s_1,\dots,s_p)F^S\in \ker\eta_F,$$
where $g(x_1,\dots,x_n,s_1,\dots,s_p)\in \mathcal{O}_X[s_1,\dots,s_p,f^{-1}]$. This means that
$$g(x_1,\dots,x_n,s,\dots,s)f^s=0,\text{ i.e. }g(x_1,\dots,x_n,s,\dots,s)=0,$$
where the last equality is an equality in $\mathcal{O}_X[s,f^{-1}]$. This means that $(X\setminus Z(f))\times\Delta\subset Z(g)$. Multiply $g$ by the smallest power of $f$ to clear denominators to obtain $h\in \mathcal{O}_X[s_1,\dots,s_p]$, now with $X\times \Delta\subset Z(h)$. The ideal defining $X\times\Delta$ in $X\times\CC^p$ is $ I_\Delta\mathcal{O}_X[s_1,\dots,s_p]$. It follows that $h\in I_\Delta\mathcal{O}_X[s_1,\dots,s_p]$, and hence $g\in I_\Delta\mathcal{O}_X[s_1,\dots,s_p,f^{-1}]$. In \eqref{prop:4} only the claim that the image is equal to $\nabla_f^k(M_f)$ requires a proof. Let $\nabla^k(P(s_1,\dots,s_p)F^s)\in \nabla_F^k(M_F)$. Now:
$$\nabla_F^k(P(s_1,\dots,s_p)F^s)=P(s_1+k,\dots,s_p+k)f^kF^{s},$$
so that by $R$-linearity of $\eta_F$,
$$\eta_F(\nabla_F^k(P(s_1,\dots,s_p)F^s))=P(s+k,\dots,s+k)f^kf^{s}=\nabla_f^k(P(s,\dots,s)f^s)\in \nabla_f^k(M_f).$$
On the other hand,
$$\nabla_f^k(P(s)f^s)=P(s+k)f^kf^s=\eta_F(P(s_1+k)f^kF^s)=\eta_F(\nabla_F^k(P(s_1)F^s)).$$
In \eqref{prop:5} the existence of $\widetilde{\eta_F}$ follows from \eqref{prop:4}. If $P(s_1,\dots,s_p)F^s\in \nabla_F(M_F)+\ker\eta_F\cap M_F$ then clearly it maps to zero under $\widetilde{\eta_F}$. On the other hand, suppose that $P(s_1,\dots,s_p)F^s$ maps to zero under $\widetilde{\eta_F}$. This means that
$$\eta_F(P(s_1,\dots,s_1)F^s)\in \nabla_f(M_f).$$
But $\nabla_F(M_F)$ surjects onto $\nabla_f(M_f)$ under $\eta_F$ by \eqref{prop:4}, so that 
$$\eta_F(P(s_1,\dots,s_1)F^s)=\eta_F(\nabla_F(Q(s_1,\dots,s_p)F^s)),$$
for some $Q(s_1,\dots,s_p)\in R$. This means that
$$P(s_1,\dots,s_p)F^s-\nabla_F(Q(s_1,\dots,s_p)F^S)=:T\in \ker \eta_F,$$
Now note that $\nabla_F(Q(s_1,\dots,s_p)F^S)\in M_F$ since $\nabla_F(M_F)\subset M_F$, so that also $T\in \ker\eta_F\cap M_F$. This concludes the proof since
 $$P(s_1,\dots,s_p)F^s=\nabla_F(Q(s_1,\dots,s_p)F^S)+T\in \nabla_F(M_F)+\ker\eta_F\cap M_F.$$
\end{proof}

\begin{proof}[proof of theorem \ref{spec}]
We continue to denote the ideal defining the diagonal in $\CC^p$ by $I_\Delta\subset\CC[s_1,\dots,s_p]$.
Consider the exact sequence
$$0\to \frac{\ker\eta_F\cap M_F+\nabla_F(M_F)}{\nabla_F(M_F)}\to \frac{M_F}{\nabla_F(M_F)}\to \frac{M_f}{\nabla_f(M_f)}\to 0,$$
where we use Proposition 5, \eqref{prop:5}. It follows from Proposition 5, \eqref{prop:3} that
$$I_\Delta\frac{M_F}{\nabla_F(M_F)}=\frac{I_\Delta M_F+\nabla_F(M_F)}{\nabla_F(M_F)}\subset \frac{\ker\eta_F\cap M_F+\nabla_F(M_F)}{\nabla_F(M_F)}.$$
We conclude that we have an exact sequence
$$0\to \frac{\ker\eta_F\cap M_F+\nabla_F(M_F)}{I_\Delta(M_F)+\nabla_F(M_F)}\to \frac{M_F}{I_\Delta(M_F)+\nabla_F(M_F)}\to \frac{M_f}{\nabla_f(M_f)}\to 0.$$
Let $b\in B_F$, and let $P\in R$ be such that
\begin{equation}bF^s=PfF^s\label{bsr}.\end{equation}
Let $T\in \ker\eta_F\cap M_F+\nabla_F(M_F)$. Write
$$T=g(x_1,\dots,x_n,s_1,\dots,s_p)f^{-k}F^s+\nabla_F(W),$$
for some $k\in \mathbb{Z}$, with the first summand in $\ker\eta_F\cap M_F$ and the second summand in $\nabla_F(M_F)$. From Proposition \ref{prop5}, \eqref{prop:3}, we know that $g\in I_\Delta \mathcal{O}_X[s_1,\dots,s_p]$. Then we observe:
\begin{align*}
\nabla_F^{-k}(b(s))T&=g(x_1,\dots,x_n,s_1,\dots,s_p)\nabla_F^{-k}(b(s))f^{-k}F^s+\nabla_F^{-k}(b(s))\nabla_F(W)\\
&=g(x_1,\dots,x_n,s_1,\dots,s_p)\nabla_F^{-k}(P)f^{-k+1}F^s+\nabla_F^{-k}(b(s))\nabla_F(W),
\end{align*}
where we use \eqref{bsr} after substituting $s=s-k$ and we abuse notation by also denoting 
$$\nabla_F:R\to R, \quad P(s_1,\dots,s_p)\mapsto P(s_1+1,\dots,s_p+1).$$
Iterating this it follows that
\begin{align*}
\prod_{i=1}^k\nabla_F^{-i}(b)T&=g(x_1,\dots,x_n,s_1,\dots,s_p)\nabla_F^{-1}(P)\dots \nabla_F^{-k}(P)F^s+\prod_{i=1}^k\nabla_F^{-i}(b)\nabla_F(W)\\
&=QF^s+\nabla_F(\prod_{i=0}^{k-1}\nabla_F^{-i}(b)W),
\end{align*}
where $Q=g(x_1,\dots,x_n,s_1,\dots,s_p)\nabla_F^{-1}(P)\dots \nabla_F^{-k}(P)\in I_\Delta R$. It follows that
$$\prod_{i=1}^k\nabla_F^{-i}(b)T\in I_\Delta(M_F)+\nabla_F(M_F).$$
Since $ \ker\eta_F\cap M_F+\nabla_F(M_F)$ is finitely generated, there exists a $k\gg 0$ such that
$$\prod_{i=1}^k\nabla_F^{-i}(b)\frac{\ker\eta_F\cap M_F+\nabla_F(M_F)}{I_\Delta(M_F)+\nabla_F(M_F)}=0.$$
Since this is true for every $b\in B_F$, we conclude that
 $$\supp_{\CC^p}\left(\frac{\ker\eta_F\cap M_F+\nabla_F(M_F)}{I_\Delta(M_F)+\nabla_F(M_F)}\right)\subset \bigcup_{i>0}\nabla^i(Z(B_F)).$$
Taking $\alpha\in (Z(B_F)\cap \Delta) \setminus  \bigcup_{i>0}\nabla^i(Z(B_F))$ we thus find
 $$0\to \frac{M_F}{\nabla_F(M_F)}\otimes_A \frac{A}{\mathfrak{m}_\alpha}\to  \frac{M_f}{\nabla_f(M_f)}\otimes_A \frac{A}{\mathfrak{m}_\alpha}\to 0,$$
 is an isomorphism. The left hand side is non-zero by Theorem \ref{relholmain}, since we know from \cite{Mai16} that $M_F/
\nabla_F(M_F)$ is relatively holonomic, which proves the claim.
\end{proof}

\smallskip
\bibliographystyle{alpha}
\bibliography{lit}
\smallskip

\end{document}